\documentclass[10pt,a4paper]{amsart}
\usepackage{amsfonts,amsmath,amssymb}
\usepackage{hyperref}
\usepackage[all]{xy}

\newtheorem*{theorem*}{Theorem}
\newtheorem*{corollary*}{Corollary}
\newtheorem{lemma}{Lemma}[subsection]
\newtheorem{proposition}[lemma]{Proposition}
\newtheorem{remark}[lemma]{Remark}

\newtheorem{theorem}[lemma]{Theorem}
\newtheorem{definition}[lemma]{Definition}
\newtheorem{notation}[lemma]{Notation}

\newtheorem{corollary}[lemma]{Corollary}
\newtheorem{fact}[lemma]{Fact}

\newtheorem*{conjecture*}{Conjecture}

\baselineskip 24pt \textwidth 15 cm  \theoremstyle{plain}
\newcommand{\tr}{\operatorname{Tr}}

\newcommand{\Hom}{\operatorname{Hom}}

\newcommand{\Spec}{\operatorname{Spec}}
\newcommand{\Lie}{\operatorname{Lie}}

\newcommand{\cc}{\mathbb{C}}

\newcommand{\eps}{\varepsilon}

\newcommand{\re}{\operatorname{Re}}

\newcommand{\Gr}{\operatorname{Gr}}

\newcommand{\N}{{\mathbb N}}

\newcommand{\Z}{{\mathbb Z}}

\newcommand{\R}{{\mathbb R}}
\newcommand{\C}{{\mathbb C}}

\newcommand{\Sc}{{\mathcal S}}

\newcommand{\Fre}{{Fr\'{e}chet \,}}

\newcommand{\tG}{{\widetilde{G}}}

\newcommand{\Fou}{{\mathcal{F}}}

\newcommand{\cD}{{\mathcal{D}}}
\newcommand{\g}{{\mathfrak{g}}}
\newcommand{\h}{{\mathfrak{h}}}

\newcommand{\Supp}{\mathrm{Supp}}

\newcommand{\cN}{{\mathcal{N}}}
\newcommand{\Zar}{{Zariski }}
\renewcommand{\sl}{{\mathrm{sl} }}

\begin{document}

\author{Avraham Aizenbud}
\address{Avraham Aizenbud and Dmitry Gourevitch, Faculty of Mathematics
and Computer Science, The Weizmann Institute of Science POB 26,
Rehovot 76100, ISRAEL.} \email{aizenr@yahoo.com}
\author{Dmitry Gourevitch} \email{guredim@yahoo.com}

\title[Multiplicity one theorem for $(\mathrm{GL}_{n+1}(\R),\mathrm{GL}_n(\R))$]
{Multiplicity one theorem for $(\mathrm{GL}_{n+1}(\R),\mathrm{GL}_n(\R))$}
\date{\today}

%
%
%
%
%
%
%
%
%
%

\begin{abstract}
Let $F$ be either $\R$ or $\C$. Consider the standard embedding
$\mathrm{GL}_n(F) \hookrightarrow \mathrm{GL}_{n+1}(F)$ and the
action of $\mathrm{GL}_n(F)$ on $\mathrm{GL}_{n+1}(F)$ by
conjugation.

In this paper we show that any $\mathrm{GL}_n(F)$-invariant
distribution on $\mathrm{GL}_{n+1}(F)$ is invariant with respect
to transposition.

We show that this implies that for any irreducible admissible
smooth \Fre representations $\pi$ of $\mathrm{GL}_{n+1}(F)$ and
$\tau$ of $\mathrm{GL}_{n}(F)$,
$$\dim \Hom_{\mathrm{GL}_n(F)}(\pi,\tau) \leq 1.$$

For p-adic fields those results were proven in \cite{AGRS}.
\end{abstract}

\maketitle
%
\setcounter{tocdepth}{3}
\tableofcontents

\section{Introduction}


Let $F$ be an archimedean local field, i.e. $F = \R$ or $F=\C$.
Consider the standard imbedding $\mathrm{GL}_n(F) \hookrightarrow
\mathrm{GL}_{n+1}(F)$. We consider the action of
$\mathrm{GL}_n(F)$ on $\mathrm{GL}_{n+1}(F)$ by conjugation. In
this paper we prove the following theorem:
\begin{theorem*} {\bf A}.
Any $\mathrm{GL}_n(F)$ - invariant
distribution on $\mathrm{GL}_{n+1}(F)$ is invariant with respect
to transposition.
\end{theorem*}

It has the following corollary in representation theory.

\begin{theorem*} {\bf B}.
Let $\pi$ be an irreducible admissible smooth \Fre representation of
$\mathrm{GL}_{n+1}(F)$ and $\tau$  be an irreducible admissible smooth \Fre representation of
$\mathrm{GL}_{n}(F)$. Then
\begin{equation}\label{dim1}
\dim \Hom_{\mathrm{GL}_n(F)}(\pi,\tau) \leq 1.
\end{equation}
\end{theorem*}

We deduce Theorem B from Theorem A using an argument due to
Gelfand and Kazhdan adapted to the archimedean case in \cite{AGS}.

Property (\ref{dim1}) is sometimes called {\it strong Gelfand
property} of the pair $(\mathrm{GL}_{n+1}(F),\mathrm{GL}_n(F))$.
It is equivalent to the fact that the pair $(\mathrm{GL}_{n+1}(F)
\times \mathrm{GL}_{n}(F), \Delta \mathrm{GL}_{n}(F))$ is a {\it
Gelfand pair}.

\begin{remark}
Using the tools developed here, combined with \cite{AGRS}, one can
easily show that Theorem A implies an analogous theorem for the
unitary groups.
\end{remark}

\begin{remark}
After the completion of this work we found out that Chen-Bo Zhu
and Sun Binyong have obtained the same results simultaneously,
independently and in a different way, see \cite{SZ}.

They also proved an analogous theorem for the orthogonal groups.
\end{remark}

\subsection{Some related results} $ $\\
For non-archimedean local fields of characteristic zero theorems A
and B were proven in \cite{AGRS}. The current paper heavily uses
the theory of D-modules, which cannot be directly applied to the
non-archimedean case. For this reason, currently there is no
uniform proof for all local fields. However, we hope that one
could develop non-archimedean analogues of the D-module techniques
that we use here.

In \cite{AGS}, a special case of  Theorem B was proven for all
local fields; namely the case when $\tau$ is one-dimensional.

Theorem A has the following corollary.

\begin{corollary*}
Let $P_n \subset \mathrm{GL}_n$ be the subgroup consisting of all
matrices whose last row is $(0,...,0,1)$. Let $ \mathrm{GL}_n$ act
on itself by conjugation. Then every $ P_n(F)$ - invariant
distribution on $ \mathrm{GL}_n(F)$ is $ \mathrm{GL}_n(F)$ -
invariant.
\end{corollary*}

This has been proven in \cite{Bar} for eigendistributions with respect to the center of $U_{\C}(\mathrm{gl}_n)$.
In \cite{Bar} it is also shown that this implies Kirillov's conjecture.

\subsection{Structure of the proof}  $ $\\
We will now briefly sketch the main ingredients of our proof of
Theorem A.

First we show that we can switch to the following problem. The
group $\mathrm{GL}_n(F)$ acts on a certain linear space $X_n$ and
$\sigma$ is an involution of $X_n$. We have to prove that every
$\mathrm{GL}_n(F)$-invariant distribution on $X_n$ is also
$\sigma$-invariant. We do that by induction on $n$. Using the
Harish-Chandra descent method we show that the induction
hypothesis implies that this holds for distributions on the
complement to a certain small closed subset $S \subset X_n$. We
call this set {\it the singular set}.

Next we assume the contrary: there exists a non-zero
$\mathrm{GL}_n(F)$-invariant distribution $\xi$ on $X$ which is
anti-invariant with respect to $\sigma$.

We use the notion of singular support of a distribution from the
theory of D-modules. Let $T \subset T^*X$ denote the singular
support of $\xi$. Using Fourier transform and the fact any such
distribution is supported in $S$ we obtain that $T$ is contained
in $\check{S}$ where $\check{S}$ is a certain small subset in
$T^*X$.

Then we use a deep result from the theory of D-modules which
states that the singular support of a distribution is a
coisotropic variety in the cotangent bundle. This enables us to
show, using a complicated but purely geometric argument, that the
support of $\xi$ is contained in a much smaller subset of $S$.

Finally it remains to prove that any $\mathrm{GL}_n(F)$-invariant
distribution that is supported on this subset together with its
Fourier transform is zero. This is proven using Homogeneity
Theorem (Theorem \ref{ArchHom}) which in turn uses Weil
representation.

\subsection{Content of the paper}$ $\\
In section \ref{Prel} we give the necessary preliminaries for the paper.

In subsection \ref{GenNat} we fix the general notation that we will use.

In subsection \ref{InvDist} we discuss invariant distributions and introduce some tools to work with them. The most advanced are
\begin{itemize}
\item The Homogeneity theorem and Fourier transform.
\item The Harish-Chandra descent method.
\end{itemize}
%

In subsection \ref{D-mod} we discuss the notion of
singular support of a distribution. The most important for us
property of this singular support is being coisotropic. This fact
is a crucial tool of this paper.

In subsection \ref{SpecNat} we introduce notation that we will use
in our proof.\\

In section \ref{HC} we use the Harish-Chandra descent method.

In subsection \ref{Lin} we linearize the problem to a problem on
the linear space $X=\sl(V)\times V \times V^*$, where $V=F^n$.

In subsection \ref{subHC} we perform the Harish-Chandra descent on
the $\sl(V)$-coordinate and $V\times V^*$ coordinate separately
and then use automorphisms $\nu_{\lambda}$ of $X$ to descend
further
to the singular set $S$.\\

In section \ref{Red2Geo} we reduce Theorem A to the following
geometric statement: any coisotropic subvariety of $\check{S}$ is
contained in a certain set $\check{C}_{X \times X}$. The reduction
is done using the fact that the singular support of a distribution
has to be coisotropic, and the following proposition: any
$\mathrm{GL}(V)$-invariant distribution on $X$ such that it and
its Fourier transform are supported on $\sl(V)\times (V \times 0
\cup 0 \times V^*)$ is zero.

In subsection \ref{ProofLemCros} we prove this proposition using
Homogeneity theorem.\\

In section \ref{ProofGeo} we prove the geometric statement. This is the most complicated part of the paper.

In subsection \ref{PrelCoisot} we give preliminaries on
coisotropic subvarieties. In particular, we give a geometric
partial analog of Frobenius reciprocity for coisotropic
subvarieties (Corollaries \ref{PreGeoFrob} and \ref{GeoFrob}).

In subsection \ref{RedKeyProp} we stratify the set $\check{S}$ and use an inductive argument on the strata. This reduces the geometric statement to a
proposition on one stratum that we call {\it the Key Proposition}.

In subsection \ref{RedKeyLem} we analyze a stratum of $\check{S}$ and then use the geometric analog of Frobenius reciprocity to reduce the Key
Proposition to a lemma on $V \times V^* \times V \times V^*$ that we call {\it the Key Lemma}.

In subsection \ref{ProofKeyLem} we prove the Key Lemma.\\

In Appendix \ref{BtoA} we prove that Theorem A implies Theorem B using an archimedean analog of Gelfand-Kazhdan technique.

In Appendix \ref{AppDmod} we give more details on the facts
concerning the theory of D-modules listed in subsection
\ref{D-mod}.
\subsection{Acknowledgements}$ $\\
We thank {\bf Joseph Bernstein} for our mathematical education. We
 thank {\bf Joseph Bernstein}, {\bf David Kazhdan}, {\bf
Bernhard Kroetz}, {\bf Eitan Sayag} and {\bf G\'{e}rard
Schiffmann} for fruitful discussions. We also thank {\bf Moshe
Baruch}, {\bf Erez Lapid} and {\bf Siddhartha Sahi} for useful
remarks.

Part of the work on this paper was done while we visited the Max
Planck Institute for Mathematics in Bonn. This visit was funded by
the Bonn International Graduate School.
%

\section{Preliminaries} \label{Prel}

\subsection{General notation} \label{GenNat}

\begin{itemize}
\item In this paper all the algebraic varieties are defined over $F$.

\item
For an algebraic variety $X$ we
denote by $X(F)$ the topological space or smooth manifold of $F$ points of $X$.

\item We consider linear spaces as algebraic varieties and treat
them in the same way.

\item For an algebraic variety $X$ we denote by $X_{\C}$ its complexification $X \times_{\Spec \R} \Spec \C$. Note that if $X$ is defined over
$\R$ then $X_{\C} \cong X \times X$.

%
\item For a group $G$ acting on a set $X$ and a point $x \in X$ we denote by $Gx$ or by $G(x)$ the orbit of $x$ and by $G_x$ the stabilizer of $x$.

\item An action of a Lie algebra $\g$ on a (smooth, algebraic, etc) manifold $M$ is a Lie algebra homomorphism from $\g$ to the Lie algebra of vector fields on $M$.
Note that an action of a (Lie, algebraic, etc) group on $M$ defines an action of its Lie algebra on $M$.

\item For a Lie algebra $\g$ acting on $M$, an element $\alpha \in \g$ and a point $x \in M$ we denote by $\alpha(x) \in T_xM$ the value at point $x$ of the vector field corresponding to $\alpha$.
We denote by $\g x \subset T_xM$ or by $\g (x)$ the image of the map $\alpha \mapsto \alpha(x)$ and by $\g_x \subset \g$ its kernel.

\item  For manifolds  $L \subset M$ we
denote by $N_L^M:=(T_M|_L)/T_L $ the normal bundle to $L$ in $M$.

\item Denote by $CN_L^M:=(N_L^M)^*$ the conormal  bundle.

\item For a point
$y\in L$ we denote by $N_{L,y}^M$ the normal space to $L$ in $M$
at the point $y$ and by $CN_{L,y}^M$ the conormal space.
\end{itemize}

\subsection{Invariant distributions} \label{InvDist}  
%
\subsubsection{Distributions on smooth manifolds} 
%

\begin{notation}
Let $X$ be a smooth manifold. Denote by $C_c^{\infty}(X)$ the space
of test functions on $X$, that is smooth compactly supported
functions, with the standard topology, i.e. the topology of
inductive limit of \Fre spaces.

Denote $\cD (X):= C_c^{\infty}(X)^*$ to be the dual space to
$C_c^{\infty}(X)$.

For any vector bundle $E$ over $X$ we denote by
$C_c^{\infty}(X,E)$ the space of smooth compactly supported
sections of $E$ and by $\cD (X,E)$ its dual space. Also, for any
finite dimensional real vector space $V$ we denote
$C_c^{\infty}(X,V):=C_c^{\infty}(X,X \times V)$ and
$\cD(X,V):=\cD(X,X\times V)$, where $X \times V$ is a trivial
bundle.
\end{notation}
\subsubsection{Schwartz distributions on Nash manifolds}
$ $\\
Our proof of Theorem A widely uses Fourier transform which cannot
be applied to general distributions. For this we require a theory
of Schwartz functions and distributions as developed in \cite{AG1}.
This theory is developed for Nash manifolds. Nash manifolds are
smooth semi-algebraic manifolds but in the present work only
smooth real algebraic manifolds are considered. Therefore the reader can
safely replace the word {\it Nash} by {\it smooth real algebraic}.

Schwartz functions are functions that decay, together with all
their derivatives, faster than any polynomial. On $\R^n$ it is the
usual notion of Schwartz function. For precise definitions of
those notions we refer the reader to \cite{AG1}. We will use the
following notations.

\begin{notation}
Let $X$ be a Nash manifold. Denote by $\Sc(X)$ the \Fre space of
Schwartz functions on $X$.

Denote by $\Sc^*(X):=\Sc(X)^*$ the space of Schwartz distributions
on $X$.

For any Nash vector bundle $E$ over $X$ we denote by $\Sc(X,E)$ the
space of Schwartz sections of $E$ and by $\Sc^*(X,E)$ its dual
space.
\end{notation}

\begin{notation}
Let $X$ be a smooth manifold and let $Z \subset X$ be a closed
subset. We denote $\Sc^*_X(Z):= \{\xi \in \Sc^*(X)|\Supp(\xi)
\subset Z\}$.

For a locally closed subset $Y \subset X$ we denote
$\Sc^*_X(Y):=\Sc^*_{X\setminus (\overline{Y} \setminus Y)}(Y)$. In
the same way, for any bundle $E$ on $X$ we define $\Sc^*_X(Y,E)$.
\end{notation}



\begin{remark}
Schwartz distributions have the following two advantages over
general distributions:\\
(i) For a Nash manifold $X$ and an open Nash submanifold $U\subset
X$, we have the following exact sequence
$$0 \to \Sc^*_X(X \setminus U)\to \Sc^*(X) \to \Sc^*(U)\to 0.$$
(ii) Fourier transform defines an isomorphism $\Fou:\Sc^*(\R^n)
\to \Sc^*(\R^n)$.
\end{remark}

The following theorem allows us to switch between general
distributions and Schwartz distributions.

\begin{theorem}   \label{NoSNoDist}
Let a reductive group $G$ act on a smooth affine variety $X$. Let
$V$ be a finite dimensional continuous representation of $G(F)$
over $\R$. Suppose that $\Sc^*(X(F),V)^{G(F)}=0$. Then
$\cD(X(F),V)^{G(F)}=0$.
\end{theorem}
For proof see \cite{AG_HC}, Theorem 4.0.8.

\subsubsection{Basic tools}
$ $\\
We present here some basic tools on  equivariant distributions that we
will use in  this paper.

\begin{proposition} \label{Strat}
Let a Nash group $G$ act on a Nash manifold $X$.
Let $Z \subset X$ be a closed subset.

Let $Z =
\bigcup_{i=0}^l Z_i$ be a Nash $G$-invariant stratification of
$Z$. Let $\chi$ be a character of $G$. Suppose that for any $k \in
\Z_{\geq 0}$ and $0 \leq i \leq l$ we have $\Sc^*(Z_i,Sym^k(CN_{Z_i}^X))^{G,\chi}=0$. Then
$\Sc^*_X(Z)^{G,\chi}=0$.
\end{proposition}

This proposition immediately follows from Corollary B.2.6 in
\cite{AGS}.

\begin{proposition} \label{Product}
Let $G_i$ be Nash groups acting on Nash manifolds $X_i$
for $i=1 \ldots n$. Let $E_i \to X_i$ be
$G_i$-equivariant Nash vector bundles. \\
(i) Suppose that $\Sc^*(X_j,E_j)^{G_j}=0$ for some $1\leq j \leq n$. Then
$$ \Sc^*(\prod_{i=1}^n X_i, \boxtimes E_i)^{\prod G_i}=0,$$
where $\boxtimes$ denotes the external product of vector
bundles.\\
(ii) Let $H_i <G_i$ be Nash subgroups. Suppose that
$\Sc^*(X_i,E_i)^{H_i}=\Sc^*(X_i,E_i)^{G_i}$ for all $i$. Then
$$\Sc^*(\prod X_i, \boxtimes E_i)^{\prod H_i}=\Sc^*(\prod X_i, \boxtimes E_i)^{\prod G_i},$$
\end{proposition}
The proof is trivial and the same as the proof of
Proposition 3.1.5 in \cite{AGS}.

\begin{theorem}[Frobenius reciprocity] \label{Frob}
Let a unimodular Nash group $G$ act transitively on a Nash
manifold $Z$. Let $\varphi:X \to Z$ be a $G$-equivariant Nash
map. Let $z\in Z$. Suppose that its stabilizer
$G_z$ is unimodular. Let $X_z$ be the fiber of $z$.
Let $\chi$ be a character of $G$. Then $\Sc^*(X)^{G,\chi}$ is
canonically isomorphic to $\Sc^*(X_z)^{G_z,\chi}$.

Moreover, for any $G$-equivariant bundle $E$ on $X$, the space
$\Sc^*(X,E)^{G,\chi}$ is canonically isomorphic to
$\Sc^*(X_z,E|_{X_z})^{G_z,\chi}$.
\end{theorem}
For proof see \cite{AG_HC}, Theorem 2.3.8.


\subsubsection{Fourier transform and Homogeneity Theorem} $ $\\
%
From now till the end of the paper we fix an additive character
$\kappa$ of $F$ given by $\kappa(x):=e^{2\pi i \re(x)}$.

\begin{notation}
Let $V$ be a vector space over $F$. Let $B$ be a non-degenerate
bilinear form on $V$. Then $V$ defines Fourier transform with respect to the
self-dual Haar measure on $V$. We denote it by $\Fou_B: \Sc^*(V) \to \Sc^*(V)$.

For any Nash manifold $M$
we also denote by $\Fou_B:\Sc^*(M \times V) \to \Sc^*(M
\times V)$ the partial Fourier transform.

If there is no ambiguity, we will write $\Fou_V$ instead $\Fou_B$.
\end{notation}

\begin{notation}
Let $V$ be a vector space over $F$. Consider the homothety action
of $F^{\times}$ on $V$ by $\rho(\lambda)v:= \lambda^{-1}v$. It
gives rise to an action $\rho$ of $F^{\times}$ on $\Sc^*(V)$.

Also, for any $\lambda \in F$ we denote $|| \lambda ||:= |\lambda|^{dim_{\R}F}$.
\end{notation}




\begin{notation}
Let $V$ be a vector space over $F$. Let $B$ be a non-degenerate
symmetric bilinear form on $V$. We denote $$Z(B):=\{x \in
V(F)|B(x,x)=0 \}.$$
\end{notation}



\begin{theorem} [Homogeneity Theorem] \label{ArchHom}
Let $V$ be a vector space over $F$. Let $B$ be a non-degenerate
symmetric bilinear form on $V$. Let $M$ be a Nash manifold. Let $L
\subset \Sc^*_{V(F)\times M}(Z(B)\times M)$ be a non-zero subspace
such that $\forall \xi \in L $ we have $\Fou_B(\xi) \in L$ and $B
\xi \in L$ (here $B$ is interpreted as a quadratic form).

Then there exist a non-zero distribution $\xi \in L$ and a unitary character $u$ of $F^{\times}$ such that either
$\rho(\lambda)\xi = || \lambda ||^{\frac{dimV}{2}} u(\lambda) \xi$ for any $\lambda \in F^{\times }$ or
$\rho(\lambda)\xi = || \lambda ||^{\frac{dimV}{2}+1} u(\lambda) \xi$ for any $\lambda \in F^{\times }$.
\end{theorem}
For proof see \cite{AG_HC}, Theorem 5.1.7.

We will also use the following trivial observation.

\begin{lemma}
Let $V$ be a finite dimensional vector space over $F$. Let a Nash
group $G$ act linearly on $V$. Let $B$ be a $G$-invariant
non-degenerate symmetric bilinear form on $V$. Let $M$ be a Nash
manifold with an action of $G$.  Let $\xi \in \Sc^*(V(F) \times
M)$ be a $G$-invariant distribution. Then $\Fou_B(\xi)$ is also
$G$-invariant.
\end{lemma}

\subsubsection{Harish-Chandra descent}

%

\begin{definition}
Let an algebraic group $G$ act on an algebraic variety $X$. We say that an element $x \in X(F)$ is {\bf $G$-semisimple} if its orbit
$G(F)x$ is closed.
\end{definition}

\begin{theorem}[Generalized Harish-Chandra descent] \label{HC_Thm}
Let a reductive group $G$ act on smooth affine varieties $X$ and
$Y$. Let $\chi$ be a character of $G(F)$. Suppose that for any
$G$-semisimple $x \in X(F)$ we have
$$\Sc^*((N_{Gx,x}^X\times Y)(F))^{G(F)_x,\chi}=0.$$ Then $\Sc^*(X(F)\times Y(F))^{G(F)_x,\chi}=0.$
\end{theorem}

For proof see \cite{AG_HC}, Theorem 3.1.6.
\subsection{D-modules and singular support} \label{D-mod} $ $\\
In this paper we will use the algebraic theory of D-modules.
We will now summarize the facts that we need and give more details in  Appendix \ref{AppDmod}.
For a good introduction to the algebraic theory of D-modules we refer the reader to \cite{Ber} and \cite{Bor}.

More specifically, we will use the notion of singular support of a
distribution. For those who are not familiar with the theory of
D-modules, Corollary \ref{CorSingSup} and the facts that are
listed after it are the only properties of singular support that
we use.

In this subsection $F=\R$.

\begin{definition}
Let $X$ be a smooth algebraic variety. Let $\xi \in \Sc^*(X(\R))$.
Consider the $D_X$-submodule ${\mathcal M}_{\xi}$ of
$\Sc^*(X(\R))$ generated by $\xi$. We define the {\bf singular
support} of $\xi$ to be the singular support of ${\mathcal
M}_{\xi}$. We denote it by $SS(\xi)$.
\end{definition}

\begin{remark}
A similar, but not equivalent notion is sometimes called in the literature a 'wave front of $\xi$'.
\end{remark}

\begin{notation}
Let $(V,B)$ be a quadratic space. Let $X$ be a smooth algebraic
variety. Consider $B$ as a map $B:V \to V^*$. Identify $T^*(X
\times V)$ with $T^*X \times V \times V^*$. We define $F_V: T^*(X
\times V) \to T^*(X  \times V)$ by $F_V(\alpha,v,\phi):= (\alpha,
-B^{-1}\phi,Bv)$.
\end{notation}

\begin{definition}
Let $M$ be a smooth algebraic variety and $\omega$ be a symplectic
form on it.
Let $Z\subset M$ be an algebraic subvariety. We call it {\bf $M$-coisotropic} if one of the following equivalent conditions holds.\\
(i) The ideal sheaf of regular functions that vanish on $\overline{Z}$ is closed under Poisson bracket. \\
(ii) At every smooth point $z \in Z$ we have  $T_zZ \supset (T_zZ)^{\bot}$. Here, $(T_zZ)^{\bot}$ denotes the orthogonal complement
to $(T_zZ)$ in $(T_zM)$ with respect to $\omega$. \\
(iii) For a generic smooth point $z \in Z$ we have $T_zZ \supset
(T_zZ)^{\bot}$.

If there is no ambiguity, we will call $Z$ a coisotropic variety.
\end{definition}
Note that every non-empty $M$-coisotropic variety is of dimension
at least $\frac{1}{2}dimM$.

\begin{notation}
For a smooth algebraic variety $X$ we always consider the standard
symplectic form on $T^*X$. Also, we denote by $p_X:T^*X \to X$ the
standard projection.
\end{notation}

The following theorem is crucial in this paper.
\begin{theorem} \label{Gaber}
Let $X$ be a smooth algebraic variety. Let ${\mathcal M}$ be a
finitely generated $D_X$-module. Then $SS({\mathcal M})$ is a
$T^*X$-coisotropic variety.
\end{theorem}
This is a special case of Theorem I in  \cite{Gab}. For similar
versions see also \cite{KKS, Mal}.

\begin{corollary} \label{CorSingSup}
Let $X$ be a smooth algebraic variety. Let $\xi \in \Sc^*(X(\R))$. Then $SS(\xi)$ is coisotropic.
\end{corollary}

We will also use the following well-known facts from the theory of
$D$-modules. Let $X$ be a smooth algebraic variety.

\begin{fact} \label{Supp2SS}
Let $\xi \in \Sc^*(X (\R))$.  Then $\overline{\Supp(\xi)}_{Zar} =
p_X(SS(\xi))(\R)$, where $\overline{\Supp(\xi)}_{Zar} $ denotes
the \Zar  closure of
 $\Supp(\xi)$.
\end{fact}
\begin{fact} \label{Ginv}
\item Let an algebraic group $G$ act on $X$. Let $\g$ denote the
Lie algebra of $G$. Let $\xi \in \Sc^*(X(\R))^{G(\R)}$. Then
$$SS(\xi) \subset \{(x,\phi) \in T^*X \, | \, \forall \alpha \in
\g \, \phi(\alpha(x)) =0\}.$$
\end{fact}
\begin{fact} \label{Fou}
Let $(V,B)$ be a quadratic space. Let $Z \subset X \times V$ be a
closed subvariety, invariant with respect to homotheties in $V$.
Suppose that $\Supp(\xi) \subset Z(\R)$. Then $SS(\Fou_V(\xi))
\subset F_V(p_{X \times V}^{-1}(Z))$.
\end{fact}
For proofs of those facts see Appendix \ref{AppDmod}.
\subsection{Specific notation} \label{SpecNat}  $ $\\
The following notations will be used in the body of the paper.
\itemize{
\item Let $V:=V_n$ be the standard $n$-dimensional linear space defined over $F$.

\item Let $\sl(V)$ denote the Lie algebra of operators with zero
trace. \item Denote $X:=X_n:=\sl(V_n) \times V_n \times V_n^*$
\item $G:=G_n:=\mathrm{GL}(V_n)$ \item $\g:=\g_n:=\Lie
(G_n)=\mathrm{gl}(V_n)$ \item $\tG:= \tG_n:= G_n \rtimes
\{1,\sigma\}$, where the action of the 2-element group
$\{1,\sigma\}$ on $G$ is given by the involution $g \mapsto
{g^t}^{-1}$. \item We define a character $\chi$ of $\tG$ by
$\chi(G)= \{1 \}$ and $\chi(\tG - G)= \{-1\}$.
\item Let $G_n$ act on $G_{n+1}$, $\g_{n+1}$ and on $\sl(V_n)$ by
$g(A):= gAg^{-1}$.
\item Let  $G$ act on $V \times V^*$ by $g(v,\phi):=(gv,
g^{*}\phi)$. This gives rise to an action of $G$ on $X$.
\item Extend the actions of $G$ to actions of $\tG$ by
$\sigma(A):=A^t$ and $\sigma(v,\phi):=(\phi^t,v^t)$. \item We
consider the standard scalar products on $\sl(V)$ and $V \times
V^*$. They give rise to a scalar product on $X$. \item We identify
the cotangent bundle $T^*X$ with $X \times X$ using the above
scalar product.
\item Let $\cN:=\cN_n \subset \sl(V_n)$ denote the cone of nilpotent operators. 
\item $C := (V \times 0) \cup (0 \times V^* ) \subset V \times
V^*$.

\item $\check{C} := (V \times 0  \times V \times 0) \cup (0 \times
V^* \times 0 \times V^*) \subset V \times V^* \times V \times
V^*$.

\item $\check{C}_{X \times X} := (\sl(V) \times V \times 0 \times
\sl(V) \times V \times 0) \cup (\sl(V) \times 0 \times V^* \times
\sl(V) \times 0 \times V^*) \subset X \times X$.

\item $S:= \{(A,v,\phi) \in X_n | A^n=0 \text{ and } \phi(A^i v)=0 \text{ for any } 0 \leq i \leq n\}$.
\item \begin{multline*}
\check{S} := \{((A_1,v_1,\phi_2),(A_2,v_2,\phi_2)) \in X
\times X \, | \, \forall i,j \in \{1,2\} \\
(A_i,v_j,\phi_j) \in S \text{ and } \forall \alpha \in
\mathrm{gl}(V), g(A_1,v_1,\phi_1) \bot (A_2,v_2,\phi_2)\}
\end{multline*}.
\item Note that
\begin{multline*}
\check{S} = \{((A_1,v_1,\phi_2),(A_2,v_2,\phi_2)) \in X \times X \, | \, \forall i,j \in \{1,2\} \\
(A_i,v_j,\phi_j) \in S \text{ and } [A_1,A_2] + v_1 \otimes \phi_2 - v_2 \otimes \phi_1 =0\}.
\end{multline*}
\item $\check{S}':= \check{S} - \check{C}_{X \times X}$.
\item $\Gamma := \{ (v,\phi) \in V \times V^* \, | \, \phi(V)=0 \}$.
\item 
For any $\lambda \in F$ we define $\nu_\lambda:X \to X$ by $\nu_{\lambda}(A,v,\phi):=(A+\lambda v\otimes \phi-\lambda \frac{\langle
\phi,v \rangle}{n}\mathrm{Id},v,\phi).$
\item It defines $\check{\nu}_\lambda : X \times X \to X \times X$. It is given by
\begin{multline*}
\check{\nu}_\lambda ((A_1,v_1,\phi_2),(A_2,v_2,\phi_2)) = \\= ((A_1+\lambda v_1\otimes \phi_1-\lambda \frac{\langle
\phi_1,v_1 \rangle}{n}\mathrm{Id},v_1,\phi_1), (A_2, v_2 -\lambda A_2 v_1,\phi_2 - \lambda A_2^* \phi_1)).
\end{multline*}
}

%
%

\section{Harish-Chandra descent} \label{HC}
\setcounter{lemma}{0}


%

\subsection{Linearization} \label{Lin}$ $\\
In this subsection we reduce Theorem A to the following one
\begin{theorem}  \label{descendant}
$\Sc^*(X(F))^{\tG,\chi}=0.$
\end{theorem}

We will divide this reduction to several propositions.\\

\begin{proposition} \label{Red1}
If  $\cD(G_{n+1}(F))^{\tG_n,\chi}=0$ then Theorem A holds.
\end{proposition}
The proof is straightforward.

\begin{proposition} 
If  $\Sc^*(G_{n+1}(F))^{\tG_n,\chi}=0$ then
$\cD(G_{n+1}(F))^{\tG_n,\chi}=0$.
\end{proposition}
Follows from Theorem \ref{NoSNoDist}.

\begin{proposition} \label{Red2}
If  $\Sc^*(\g_{n+1}(F))^{\tG_n(F),\chi}=0$ then
$\Sc^*(G_{n+1}(F))^{\tG_n(F),\chi}=0$.
\end{proposition}
\begin{proof} Let $\xi \in \Sc^*(G_{n+1}(F))^{\tG_n(F),\chi}$. We have to
prove $\xi=0$. Assume the contrary. Take $p \in
\mathrm{Supp}(\xi)$. Let $t=\mathrm{det}(p)$. Let $f\in \Sc(F)$ be
such that $f$ vanishes in a neighborhood of 0 and $f(t) \neq 0$.
Consider the determinant map $\mathrm{det}:G_{n+1}(F) \to F$.
Consider $\xi':=(f \circ \mathrm{det})\cdot \xi$. It is easy to
check that $\xi' \in \Sc^*(G_{n+1}(F))^{\tG_n(F),\chi}$ and $p \in
\mathrm{Supp}(\xi')$. However, we can extend $\xi'$ by zero to
$\xi'' \in \Sc^*(\g_{n+1}(F))^{\tG_n(F),\chi}$, which is zero by
the assumption. Hence $\xi'$ is also zero. Contradiction.
\end{proof}

\begin{proposition}
If $\Sc^*(X_{n}(F))^{\tG_n(F),\chi}=0$ then
$\Sc^*(\g_{n+1}(F))^{\tG_n(F),\chi}=0$.
\end{proposition}

\begin{proof}
The $\tG_n(F)$-space $\mathrm{gl}_{n+1}(F)$ is isomorphic to
$X_n(F) \times F \times F$ with trivial action on $F\times F$.
This isomorphism is given by
$$ \left(
  \begin{array}{cc}
    A_{n \times n} & v_{n\times 1} \\
    \phi_{1\times n} & \lambda \\
  \end{array}
\right) \mapsto ((A - \frac{\tr A}{n} \mathrm{Id}, v
,\phi),\lambda, \tr A).$$
 \end{proof}

\subsection{Harish-Chandra descent} \label{subHC}$ $\\

Now we start to prove Theorem \ref{descendant}. The proof is by
induction on $n$. Till the end of the paper we will assume that
Theorem \ref{descendant} holds for all $k<n$ for both archimedean
local fields.

The theorem obviously holds for $n=0$. Thus from now on we assume
$n\geq 1$. The goal of this subsection is to prove the following
theorem.

\begin{theorem}  \label{LastRed}
$\Sc^*(X(F) - S(F))^{\tG,\chi}=0.$
\end{theorem}

In fact, one can prove this theorem directly using  Theorem
\ref{HC_Thm}. However, this will require long computations. Thus,
we will
 divide the proof to several steps and use some
tricks to avoid part of those computations.

\begin{proposition}  
$\Sc^*(X(F) - (\cN \times V \times V^*)(F))^{\tG,\chi}=0.$
\end{proposition}
\begin{proof}

By Theorem \ref{HC_Thm} it is enough to prove that for any
semi-simple $A \in \sl(V)$ we have
$$\Sc^*((N_{GA,A}^{\sl(V)} \times (V \times V^*)) (F))^{\tG_A,\chi} = 0.$$

Now note that $\tG_A(F) \cong \prod \tG_{n_i}(F_i)$ where $n_i<n$
and $F_i$ are some field extensions of $F$.
Note also that
$$(N_{GA,A}^{\sl(V)} \times V \times V^*)(F) \cong
\sl(V)_{A} \times (V \times V^*)(F) \cong \prod X_{n_i}(F_i)
\times {\mathcal Z}(\sl(V)_{A})(F),$$ where ${\mathcal
Z}(\sl(V)_{A})$ is the center of $\sl(V)_{A}$. Clearly, $\tG_A$
acts trivially on ${\mathcal Z}(\sl(V)_{A})$.

Now by Proposition \ref{Product} the induction hypothesis implies
that
$$\Sc^*(\prod
X_{n_i}(F_i) \times {\mathcal Z}(\sl(V)_{A})(F))^{\prod
\tG_{n_i}(F_i),\chi} = 0.$$
\end{proof}

In the same way we obtain the following proposition.

\begin{proposition}  
$\Sc^*(X(F) - (\sl(V) \times \Gamma)(F))^{\tG,\chi}=0.$
\end{proposition}

\begin{corollary}  \label{PartDes}
$\Sc^*(X(F) - (\cN \times \Gamma)(F))^{\tG,\chi}=0.$
\end{corollary}

%
\begin{lemma}
Let $A \in \sl(V)$, $v \in V$ and $\phi \in V^*$. Suppose $A +
\lambda v \otimes \phi$ is nilpotent for all $\lambda \in F$. Then
$\phi(A^i v)=0$ for any $i \geq 0$.
\end{lemma}
\begin{proof}
Since $A + \lambda v \otimes \phi$ is nilpotent, we have $tr(A+
\lambda v \otimes \phi)^k=0$ for any $k \geq 0$ and $\lambda \in
F$. By induction on $i$ this implies that $\phi(A^i v)=0$.
\end{proof}

\begin{proof}[Proof of Theorem \ref{LastRed}]
By the previous lemma, $\bigcap_{\lambda \in F}\nu_{\lambda}(\cN
\times \Gamma) \subset S$. Hence $\bigcup _{\lambda \in
F}\nu_{\lambda}(X - \cN \times \Gamma) \supset X-S$.

By Corollary \ref{PartDes} $\Sc^*(X(F) - (\cN \times
\Gamma)(F))^{\tG,\chi}=0$. Note that $\nu_{\lambda}$ commutes with
the action of $\tG$. Thus $\Sc^*(\nu_{\lambda}(X(F) - (\cN \times
\Gamma)(F)))^{\tG,\chi}=0$ and hence $\Sc^*(X(F) -
S(F))^{\tG,\chi}=0.$
\end{proof}

\section{Reduction to the geometric statement} \label{Red2Geo}
\setcounter{lemma}{0}

In this section coisotropic variety means $X \times X$-coisotropic variety.

The goal of this section  is to reduce Theorem \ref{descendant} to
the following geometric statement.
\begin{theorem}[geometric statement] \label{GeoStat}
For any coisotropic subvariety of $T \subset \check{S}$ we have
$T\subset \check{C}_{X \times X}$.
\end{theorem}

Till the end of this section we will assume the geometric
statement.

\begin{proposition}
Let $\xi \in \Sc^*(X(F))^{\tG,\chi}=0.$ Then $\Supp(\xi) \subset
(\sl(V) \times C)(F)$.
\end{proposition}
\begin{proof}[Proof for the case $F=\R$] $ $

Step 1. $SS(\xi) \subset \check{S}$.\\
We know that $$\Supp(\xi), \Supp(\Fou^{-1}_{\sl(V)}\xi),
\Supp(\Fou^{-1}_{V \times V^*}(\xi)), \Supp(\Fou^{-1}_{X}(\xi))
\subset S(F).$$
By Fact \ref{Fou} this implies that $$SS(\xi) \subset (S \times X)
\cap F_{\sl(V)}(S \times X) \cap F_{V\times V^*}(S \times X) \cap
F_{X}(S \times X).$$ On the other hand, $\xi$ is $G$-invariant and
hence by Fact \ref{Ginv} $$SS(\xi) \subset \{((x_1,x_2) \in X
\times X \, | \, \forall g \in \g, g(x_1) \bot x_2 \}. $$ Thus
$SS(\xi)\subset \check{S}$.

Step 2. $SS(\xi) \subset \check{C}_{X \times X}$.\\
By Corollary \ref{Gaber}, $SS(\xi)$ is $X \times X$-coisotropic
and hence by the geometric statement $SS(\xi) \subset \check{C}_{X
\times X}$.

Step 3. $\Supp(\xi) \subset (\sl(V) \times C)(F)$.\\
Follows from the previous step by Fact \ref{Supp2SS}.
\end{proof}

The case $F=\C$ is proven in the same way using the following corollary of the geometric statement.
\begin{proposition}
Any $(X \times X)_{\C}$-coisotropic subvariety of $\check{S}_{\C}$
is contained in $(\check{C}_{X \times X})_{\C}$.
\end{proposition}

Now it is left to prove the following proposition.

\begin{proposition} \label{Cross}
Let $\xi \in \Sc^*(X(F))^{\tG(F),\chi}$ be such that
$$\Supp(\xi),\Supp(\Fou_{V \times V^*}(\xi)) \subset (\sl(V) \times C)(F).$$
Then $\xi=0$.
\end{proposition}

\subsection{Proof of proposition \ref{Cross}} \label{ProofLemCros} $ $\\
Proposition \ref{Cross} follows from the following lemma.

\begin{lemma}
Let $F^{\times}$ act on $V\times V^*$ by $\lambda(v,\phi) :=
(\lambda v, \frac{\phi}{\lambda})$.
 Let $\xi \in
\Sc^*((V\times V^*)(F))^{F^{\times}}$ be such that $$\Supp(\xi),
\Supp(\Fou_{V \times V^*}(\xi)) \subset C(F).$$ Then $\xi=0$.
\end{lemma}

By Homogeneity Theorem (Theorem \ref{ArchHom}) it is enough to
prove the following lemma.

\begin{lemma}
Let $\mu$ be a character of $F^{\times}$ given by $|| \cdot ||^n
u$ or $|| \cdot ||^{n+1} u$ where $u$ is some unitary character.
Let $F^{\times} \times F^{\times}$ act on $V \times V^*$ by
$(x,y)(v,\phi) = (\frac{y}{x} v, \frac{1}{xy} \phi)$. Then
$\Sc^*_{(V\times V^*)(F)}(C(F))^{F^{\times}\times F^{\times}, \mu
\times 1 }=0.$
\end{lemma}
By Proposition \ref{Strat} this lemma follows from the following
one.

\begin{lemma}
For any $k \geq 0$ we have\\
(i) $\Sc^*(((V-0) \times 0) (F), Sym^k(CN_{(V-0) \times
0}^{V\times V^*}(F)))^{F^{\times}\times F^{\times}, \mu
\times 1 } = 0$.\\
(ii) $\Sc^*((0 \times (V^*-0)) (F), Sym^k(CN_{0 \times
(V^*-0)}^{V\times V^*}(F)))^{F^{\times}\times F^{\times}, \mu
\times 1 } = 0$.\\
(iii) $\Sc^*(0, Sym^k(CN_{0}^{V\times V^*}(F)))^{F^{\times}\times
F^{\times}, \mu
\times 1 } = 0$.\\
\end{lemma}

\begin{proof}$ $\\
(i) Cover $V-0$ by standard affine open sets $V_i:=\{x_i \neq
0\}$. It is enough to show that $\Sc^*((V_i \times 0) (F),
Sym^k(CN_{(V_i \times 0)(F)}^{V\times V^*}(F)))^{F^{\times}\times
F^{\times}, \mu \times 1 } = 0$.

Note that $V_i$ is isomorphic as an $F^{\times} \times
F^{\times}$-manifold to $F^{n-1} \times F^{\times}$ with the
action given by $(x,y)(v,\alpha) = (v,\frac{y}{x}\alpha)$. Note
also that the bundle $Sym^k(CN_{(V_i \times 0)(F)}^{V\times
V^*}(F))$ is a constant bundle with fiber $Sym^k(V)$.

Hence by Proposition \ref{Product} it is enough to show that
$\Sc^*(F^{\times}, Sym^k(V))^{F^{\times} \times F^{\times}, \mu
\times 1}=0$. Let $H := (F^{\times} \times F^{\times})_{1} = \{(t,
t) \in  F^{\times} \times F^{\times} \}$. Now by Frobenius
reciprocity (theorem \ref{Frob}) it is enough to show that
$(Sym^k(V^*(F))\otimes_{\R}\C) ^{H,\mu \times 1|_H}=0$. This is
clear since $(t,t)$ acts on $(Sym^k(V^*(F))$ by multiplication
by $t^{-2k}$.\\
(ii) is proven in the same way.\\
(iii) is equivalent to the statement $((Sym^k(V\times V^*)(F))
\otimes_{\R}\C) ^{F^{\times}\times F^{\times}, \mu \times 1 } =
0$. This is clear since $(t,1)$ acts on $Sym^k(V\times V^*)(F)$ by
multiplication by $t^{-k}$.
\end{proof}

\section{Proof of the geometric statement} \label{ProofGeo}
\subsection{Preliminaries on coisotropic subvarieties} \label{PrelCoisot}

\begin{proposition}
Let $M$ be a smooth algebraic variety with a symplectic form on
it. Let $R \subset M$ be an algebraic subvariety. Then there
exists a maximal $M$-coisotropic subvariety of $R$ i.e. an
$M$-coisotropic subvariety $T \subset M$ that includes all
$M$-coisotropic subvarieties of $R$.
\end{proposition}
\begin{proof}
Let $T'$ be the union of all smooth $M$-coisotropic subvarieties
of $R$. Let $T$ be the Zariski closure of $T'$ in $R$. Clearly,
$T$ includes all $M$-coisotropic subvarieties of $R$.  Let $U$
denote the set of regular points of $T$. Clearly $U \cap T'$ is
dense in $U$. On the other hand, for any $x \in U \cap T'$, the
tangent space to $T$ at $x$ is coisotropic. Hence $T$ is
coisotropic.
\end{proof}

\begin{remark}
Suppose $M$ is affine. Then $T$ can be computed explicitly in the
following way. Let $I$ be the ideal of regular functions that
vanish on $\overline{R}$. We can iteratively close it with respect
to Poisson brackets and taking radical. Since ${\mathcal O}(M)$ is
Noetherian, this process will stabilize. Let $J$ denote the
obtained closure and $Z(J)$ denote the zero set of $J$. Then
$T=Z(J) \cap R$.
\end{remark}

The following lemma is trivial.
\begin{lemma}
Let $M$ be a smooth algebraic variety and $\omega$ be a symplectic
form on it. Let a group $G$ act on $M$ preserving $\omega$. Let
$S$ be a $G$ -invariant subvariety. Then the maximal
$M$-coisotropic subvariety of $S$ is also $G$-invariant.
\end{lemma}

\begin{definition}
Let $Y$ be a smooth algebraic variety. Let $Z \subset Y$ be a smooth subvariety and $R \subset T^*Y$  be any subvariety. We define {\bf the
restriction $R|_Z \subset T^*Z$ of $R$ to $Z$ } in the following way.
Let $R' = p_Y^{-1}(Z) \cap R$. Let $q:p_Y^{-1}(Z) \to T^*Z$ be the
projection. We define $R|_Z:=q(R')$.
\end{definition}

\begin{lemma} \label{Restriction}
Let $Y$ be a smooth algebraic variety. Let $Z \subset Y$ be a
smooth subvariety and $R \subset T^*Y$  be a coisotropic
subvariety.
Assume that any smooth point $z \in p_Y^{-1}(Z) \cap R$ is also a
smooth point of $R$ and we have $T_z(p_Y^{-1}(Z) \cap R) =
T_z(p_Y^{-1}(Z)) \cap T_zR$.

Then $R|_Z$ is $T^*Z$ coisotropic.
\end{lemma}

In the proof we will use the following straightforward lemma.
\begin{lemma}
Let $W$ be a linear space. Let $L \subset W$ be a linear subspace and $R \subset W \oplus W^*$ be a coisotropic subspace. Then $R|_L$ is
$L \oplus L^*$ coisotropic.
\end{lemma}

\begin{proof}[Proof of lemma \ref{Restriction}]
Without loss of generality we assume that $R$  is irreducible. Let
$R' = p_Y^{-1}(Z) \cap R$. Without loss of generality we assume
that $R'$  is irreducible. Let $R''$ be the set of smooth points
of $R'$. Let $q:p_Y^{-1}(Z) \to T^*Z$ be the projection. Let
$R'''$ be the set of smooth points in $q(R'')$. Clearly $R'''$ is
dense in $R|_Z$. Hence it is enough to prove that for any $x \in
R'''$ the space $T_x(R|_Z)$ is coisotropic. Let $y \in R''$ s.t.
$x=q(y)$. Denote $W:=T_{p_Y(y)}Y$. Let $Q:=T_yR \subset W \oplus
W^*$. Let $L:=T_{p_Y(y)}Z$. By the assumption $T_x(R|_Z) \supset
Q|_L$. By the lemma, $Q|_L$ is coisotropic and hence $T_x(R|_Z)$
is also coisotropic.
\end{proof}

\begin{corollary} \label{PreGeoFrob}
Let $Y$ be a smooth algebraic variety.  Let an algebraic group $H$
act on $Y$. Let $q:Y \to B$ be an $H$-equivariant morphism. Let $O
\subset B$ be an orbit. Consider the natural action of $G$ on
$T^*Y$ and let $R \subset T^*Y$ be an $H$-invariant subvariety.
Suppose that $p_Y(R) \subset q^{-1}(O)$. Let $x \in O$. Denote
$Y_x:= q^{-1}(x)$. Then

\itemize{ \item if $R$ is $T^*Y$-coisotropic then $R|_{Y_x}$ is
$T^*(Y_x)$-coisotropic.}
\end{corollary}

\begin{corollary} \label{GeoFrob}
In the notation of the previous corollary, if $R|_{Y_x}$ has no
(non-empty) $T^*(Y_x)$-coisotropic subvarieties then $R$ has no
(non-empty) $T^*(Y)$-coisotropic subvarieties.
\end{corollary}

Note that the converse statement does not hold in general.
\subsection{Reduction to the Key Proposition} \label{RedKeyProp}
$ $\\
In this subsection coisotropic variety means $X \times X$-coisotropic variety.

We will use the following notation.
\begin{notation}$ $\\
(i) For any nilpotent operator $A\in \sl(V)$ we denote
$$Q_A:= \{(v,\phi) \in V \times V^* \, | \, v \otimes \phi \in [A,\g]\}=\{(v,\phi) \in V
\times V^* \, | \, (v \otimes \phi) \bot \g_A\}.$$ (ii) Denote by
$T$ the maximal coisotropic subvariety of $\check{S}'$.\\
(iii) For any two nilpotent orbits $O_1,O_2 \subset N$  denote
\begin{multline*}
U(O_1,O_2):=  \{(A_1,v_1,\phi_1,A_2,v_2,\phi_2) \in X \times X| \, \forall i,j \in \{1,2\} \\
A_i \in O_i ,(v_j,\phi_j) \in Q_{A_i}, \, [A_1,A_2] + v_1 \otimes
\phi_2 - v_2 \otimes \phi_1 = 0, (v_1,\phi_1,v_2,\phi_2) \notin
\check{C} \}.
\end{multline*}
\end{notation}

The geometric statement is equivalent to the following theorem
\begin{theorem}
$T = \emptyset.$
\end{theorem}

The goal of this subsection is to reduce the geometric statement
to the following Key Proposition.
\begin{proposition}[Key Proposition] \label{KeyProp}
For any two nilpotent orbits $O_1,O_2$ there are no (non-empty)
coisotropic subvarieties in $U(O_1,O_2)$.
\end{proposition}
The reduction will be in the spirit of the beginning of section 3 in \cite{AGRS}.

\begin{notation}
Let $$\cN^i =\{(A_1,A_2) \in \cN \times \cN| \dim G(A_1) + \dim
G(A_2) \leq i \}.$$

$$\widehat{\cN^i} :=\{(A_1,v_1,\phi_1,A_2,v_2,\phi_2) \in \check{S}' | (A_1,A_2) \in \cN^i\}.$$
\end{notation}

We will prove by descending induction that $T \subset
\widehat{\cN^i}$. From now on we fix $i$, suppose that this holds
for $i$ and prove that holds for $i-1$. Let $\mathfrak{S}$ denote
the subgroup of automorphisms of $X \times X$ generated by
$\check{\nu}_{\lambda}$, $F_{\sl(V)}$ and $F_{V \times V^*}$.

Denote $\widetilde{\cN^i}:= \bigcap_{\nu \in \mathfrak{S}}
\nu(\widehat{\cN^i})$. We know that $T \subset \widehat{\cN^i}$,
and hence $T \subset \widetilde{\cN^i}$. Let $U^i:=
\widetilde{\cN^i} - \widehat{\cN^{i-1}}$. It is enough to show
that $U^i$ does not have (non-empty) coisotropic subvarieties.

\begin{notation}
Let $O_1,O_2$ be nilpotent orbits such that $\dim O_1+\dim O_2=i$.
Denote $U'(O_1,O_2):=\{(A_1,v_1,\phi_1,A_2,v_2,\phi_2) \in U^i|
A_1\in O_1,\, A_2 \in O_2\}$.
\end{notation}

Since the sets $U'(O_1,O_2)$ form an open cover of $U^i$, it is
enough to show that each $U'(O_1,O_2)$ does not have (non-empty)
coisotropic subvarieties. This fact clearly follows from the Key
Proposition using the following easy lemma.

\begin{lemma}
$U'(O_1,O_2) \subset U(O_1,O_2)$.
\end{lemma}

\subsection{Reduction to the Key Lemma} \label{RedKeyLem}
$ $\\
We will use the following notation
\begin{notation}
\begin{multline*}
R_A := \{ (v_1,\phi_1, v_2, \phi_2) \in Q_A  \times Q_A - \check{C} \, |\\
\exists B \in [A ,\g] \cap \cN \text{ such that } [A,B] + v_1
\otimes \phi_2 - v_2 \otimes \phi_1 = 0 \}.
\end{multline*}
\end{notation}

The goal of this subsection  is to reduce the Key Proposition to
the following Key Lemma.
\begin{lemma}[Key Lemma] \label{KeyLem}
$R_A$ does not have (non-empty) $V \times V^* \times V \times
V^*$-coisotropic subvarieties.
\end{lemma}


\begin{notation}
Denote
$$U''(O_1,O_2):=\{(A_1,v_1,\phi_1,A_2,v_2,\phi_2) \in U(O_1,O_2) | \g_{A_1} \bot \g_{A_2} \}.$$
\end{notation}

\begin{lemma}
Any $X \times X$-coisotropic subvariety of $U(O_1,O_2)$ lies in $U''(O_1,O_2)$.
\end{lemma}
\begin{proof}
Denote $M:=O_1 \times V \times V^* \times O_2 \times V \times
V^*$. Note that $U(O_1,O_2) \subset M$. Note that any coisotropic
subvariety of $M$ is contained in
$M':=\{(A_1,v_1,\phi_1,A_2,v_2,\phi_2) \in M \,| \, \g_{A_1} \bot
\g_{A_2}\}$. Hence any coisotropic subvariety of $U(O_1,O_2)$ is
contained in $U(O_1,O_2) \cap M'$.
\end{proof}


The following straightforward lemma together with Corollary
\ref{GeoFrob} finish the reduction.
\begin{lemma}
$U''(O_1,O_2)|_{A \times V \times V^*} \subset R_A.$
\end{lemma}


\subsection{Proof of the Key Lemma} \label{ProofKeyLem}$ $\\
We will first give a short description of the proof for the case
when $A$ is one Jordan block. Then we will present the proof in
the general case.

During the whole subsection coisotropic variety means $V \times
V^* \times V \times V^*$-coisotropic variety.
\subsubsection{Proof in the case when $A$ is one Jordan block} $ $\\
In this case $Q_A = \bigcup_{i=0}^n (Ker A^i) \times (Ker
(A^*)^{n-i})$. Hence $$Q_A \times Q_A = \bigcup_{i,j=0}^n
 (Ker A^i) \times (Ker (A^*)^{n-i}) \times (Ker
A^j) \times (Ker (A^*)^{n-j}).$$
Denote $L_{ij} := (Ker A^i) \times (Ker (A^*)^{n-i}) \times (Ker
A^j) \times (Ker (A^*)^{n-j})$.

It is easy to see that any coisotropic subvariety of $Q_A \times
Q_A$ is contained in $\bigcup_{i=0}^n L_{ii}.$ Hence it is enough
to show that for any $i$, $\dim R_A \cap L_{ii} < 2n$. For $i=0,n$
it is clear since $R_A \cap L_{ii}$ is empty. So we will assume
$0<i <n$.

Let $f \in {\mathcal O}(L_{ii})$ be the polynomial defined by
$f(v_1,\phi_1,v_2,\phi_2):= (v_1)_{i} (\phi_2)_{i+1} - (v_2)_{i}
(\phi_1)_{i+1}$, where $( \cdot)_i$ means the i-th coordinate. It
is enough to show that $f(R_A \cap L_{ii}) = \{0\}$.

Let $(v_1,\phi_1,v_2,\phi_2) \in L_{ii}$. Let $M:= v_1 \otimes
\phi_2 - v_2 \otimes \phi_1$. Clearly, $M$ is of the form
$$ M= \begin{pmatrix}
  &0_{i \times i} &* \\
  & 0_{(n-i) \times i} &0_{(n-i) \times (n-i)}
\end{pmatrix}.  $$
Note also that $M_{i,i+1}=f(v_1,\phi_1,v_2,\phi_2)$.

It is easy to see that any $B$ satisfying $[A,B]=M$ is upper
triangular. On the other hand, we know that there exists a
nilpotent $B$ satisfying $[A,B]=M$. Hence this $B$ is upper
nilpotent, which implies $M_{i,i+1}=0$ and hence
$f(v_1,\phi_1,v_2,\phi_2)=0$.

\subsubsection{Notation on filtrations}

\begin{notation}$ $\\
(i) Let $L$ be a vector space with a gradation $G^iL$. It defines a filtration $G^{\geq i}L$ by
$G^{\geq i}L := \bigoplus_{j \geq i} G^jL$ .\\
(ii) Let $L$ be a vector space with a descending filtration
$F^{\geq i}$. We define $F^{>i}L := \bigcup_{j>i} F^{\geq j}L$.

\end{notation}

\begin{notation}
Let $L$ and $M$ be vector spaces with  descending filtrations $F^{\geq i}L$ and $F^{\geq i}M$.

Define filtrations $F^{\geq i}(L \otimes M):= \sum_{k+l=i} F^{\geq
k}L \otimes F ^{\geq l}M$ and $F^{\geq
i}(L^*):=(F^{>-i}L)^{\bot}$.

Similarly for gradations $G^iL$ and $G^iM$ we define gradations
$G^i(L \oplus M):= \bigoplus_{k+l=i} G^kL \otimes G^lM$  and
$G^i(L^*):= (\bigoplus_{j\neq -i} G^{j}L)^{\bot}$.
\end{notation}

We fix a standard basis $\{E,H,F\}$ of $\sl_2$.

\begin{notation}
Let $L$ be a representation of $\sl_2$. We define
\begin{itemize}
\item A gradation $W^{\alpha}(L):= Ker(H - \alpha Id)$ and
\item An ascending  filtration $K_i(L) := Ker(E^i)$.
\end{itemize}
\end{notation}

Note that if $L$ is an irreducible representation then $K_i(L)=
W^{\geq dimL+1-2i}(L)$.

\subsubsection{Proof of the Key Lemma}$ $\\
We will cover $R_A$ by linear spaces and show that every one of
them does not include coisotropic subvarieties of $R_A$.

Fix a morphism of Lie algebras $\psi: \sl_2 \to \sl(V)$ such that
$\psi(E)=A$. Decompose $V$ to irreducible representations of
$\sl_2$: $V = \bigoplus_{i=1}^k V_i$ such that $dim V_i \geq \dim
V_{i+1}$.
\begin{notation}
Denote $D_i := \dim V_i$. Let $D$ denote the multiindex
$D:=(D_1,...,D_k)$.

For any multiindex $I=(I_1,...,I_k)$ such that $0 \leq I_l \leq D_l$, $I \neq 0$ and $I\neq D$  we define
\itemize{
\item $L_I := W^{\geq D_1+1-2I_1}(V_1) \oplus ... \oplus W^{\geq D_k+1-2I_k}(V_k)= K_{I_1}(V_1) \oplus ... \oplus K_{I_k}(V_k)$
\item $L'_I:= W^{\geq D_1+1-2I_1}(V^*_1) \oplus ... \oplus W^{\geq D_k+1-2I_k}(V^*_k)= K_{I_1}(V^*_1) \oplus ... \oplus K_{I_k}(V^*_k)$
\item $L_{IJ}:= L_I \times L'_{D-I} \times L_J \times L'_{D-J}$}
\end{notation}

The following two lemmas are straightforward

\begin{lemma}
$$ R_A \subset \bigcup_{I,J} L_{IJ}$$
\end{lemma}

\begin{lemma}
$L_{IJ}$ is not coisotropic if $I \neq J$.
\end{lemma}

Hence it is enough to prove the following proposition.

\begin{proposition} \label{enough}
$\dim L_{II}\cap R_A < 2n$.
\end{proposition}

From here on we fix $I$ and suppose that the proposition does not
hold for this $I$. Our aim now is to get a contradiction. Note
that if Proposition \ref{enough} holds for $I$ then it holds for
$D-I$. Hence without loss of generality we can (and will) assume
$I_k<D_k$.

\begin{lemma} \label{Ineq}
For any $m<l$ we have $D_m -D_l \geq I_m-I_l \geq 0$.
\end{lemma}

Before we prove this lemma we introduce some notation.

We fix a Jordan basis of $A$ in each $V_i$.
\begin{notation}
For any $v \in V,\phi \in V^*,X \in V \otimes V^*$ we define $v^l$
to be the $l$-th component of $v$ with respect to the
decomposition $V = \oplus V_l$ and $v^l_{\alpha}$ to be its
$\alpha$ coordinate.

Similarly we define $\phi^l,\phi^l_{\alpha},X^{lm},X^{lm}_{\alpha,\beta}$
\end{notation}
\begin{proof}[Proof of lemma \ref{Ineq}]
It is enough to prove that for any $l,m$ we have $I_l + (D_m-I_m)
\leq max(D_l,D_m)$. Assume that the contrary holds for some $l,m$.
It is enough to show that in this case $\dim(Q_A \cap (L_I \times
L'_{D-I}))<n$. Consider the function $g \in {\mathcal O}(L_I
\times L'_{D-I})$ defined by $g(v,\phi) = \phi^m_{I_m+1} \cdot
v^l_{I_l}$. It is enough to show that $g(Q_A \cap (L_I \times
L_{D-I}) )=\{0\}$.

Let $B \in V_m \otimes V_l^*$ be defined by $B_{\alpha, \beta}=
\delta_{\alpha-\beta,I_m-I_l+1}$. Consider $B$ as an element of
$\g$.
Note that $B \in \g_A$ and $\langle B , v \otimes \phi \rangle =
g(v,\phi)$ for any $(v,\phi) \in L_I \times L'_{D-I}$. Hence
$g(Q_A \cap (L_I \times L_{D-I}) )=\{0\}$.
\end{proof}

\begin{corollary} $ $\\
(i) If $I_m=0$ then $I_l=0$ for any $l>m$.\\
(ii) If $I_m=D_m$ then $I_l=D_l$ for any $l>m$.
\end{corollary}

\begin{corollary}
$D_1>I_1>0$.
\end{corollary}

\begin{notation}
Let $k'$ be the maximal index such that $D_{k'}>I_{k'} > 0$.
\end{notation}


\begin{notation}
Define $f_l \in {\mathcal O}(V\times V^* \times V \times V^*)$ by
$$f_l(v_1,\phi_1,v_2,\phi_2) := (v_1)^l_{I_l} (\phi_2)^l_{I_l+1} -
(v_2)^l_{I_l} (\phi_1)^l_{I_l+1}.$$
Define also $f :=\sum_{l=1}^{k'} \frac{D_l-I_l}{D_l} f_l$.
\end{notation}

Now it is enough to prove the following proposition.

\begin{proposition} \label{mekorit}
$$f(R_A \cap L_{II})= \{0 \}.$$
\end{proposition}

We will need several notations and straightforward lemmas.

\begin{lemma} \label{EX}
For any $\alpha \leq |D_m-D_l|$ we have $W^{\geq \alpha}(V_l
\otimes V^*_m)= \{X \in V_l \otimes V_m^*| E(X) \in W^{\geq
\alpha+2}(V_l \otimes V^*_m) \}$.
\end{lemma}

\begin{definition}
Define gradation $W_I^i$ on $V_l$  by
$W^i_I(V_l)=W^{i+(D_l+l-2I_l)}(V_l)$. It gives rise to gradations
$W_I^{i}$ on $V_l^*,V_m \otimes V_l^*, V, V^*$.
\end{definition}

\begin{lemma}$ $\\
(i) If $i$ is odd then $W_I^{i}=0$.\\
(ii) $W_I^{\geq 0}(V) = L_I$.\\
(iii) $W_I^{\geq 2}(V^*) = L'_{D-I}$.
\end{lemma}

\begin{definition}
Let ${\mathcal A}$ be the algebra $W^{\geq 0}_I(V \otimes V^*)$
and ${\mathcal I}$ be its ideal $W^{>0}_I(V \otimes V^*)= W^{\geq
2}_I(V \otimes V^*)$. Clearly ${\mathcal A}/{\mathcal I} \cong
\prod End(W^i_I(V))$.  This gives rise to a homomorphism
$\eps:{\mathcal A} \to End(W^0_I(V))$.
\end{definition}

\begin{lemma}$ $\\
(i) ${\mathcal A} = \bigoplus_{1 \leq l,m \leq k} W^{\geq D_l-D_m - 2(I_l-I_m)}(V_l \otimes V^*_m)$.  \\
(ii) ${\mathcal I}:= \bigoplus_{1 \leq l,m \leq k}W^{\geq D_l-D_m - 2(I_l-I_m)+2}(V_l \otimes V^*_m)$.\\
(iii) $dim (W^0_I(V))=k'$ \\
(iv) Consider the basis on $W^0_I(V)$ corresponding to the one on
$V$ and identify $End(W^0_I(V))$ with $\mathrm{gl}(k')$. Then
$$\eps(X)_{lm}:= X_{I_l,I_m}^{lm}.$$
\end{lemma}

\begin{corollary} \label{AshiftI}
${\mathcal A}= \{X \in End(V)| [A,X] \in {\mathcal I}\}$.
\end{corollary}
\begin{proof}
Follows from the previous lemma using Lemma \ref{EX}.
\end{proof}

\begin{proof}[Proof of Proposition \ref{mekorit}]
Let $(v_1,\phi_1,v_2,\phi_2)  \in L_{II} \cap R_A$. Let $M:= v_1
\otimes \phi_2 - v_2 \otimes \phi_1$. We know that there exists a
nilpotent matrix $B \in [A,\mathrm{gl}(V)]$ such that $[A,B]=M$.
By Corollary \ref{AshiftI} $B \in \mathcal{A}$. Denote $\Delta:=
\eps(B)$. Fix $1 \leq l \leq k'$. Denote
$a_{l}:=M^{ll}_{I_l,I_l+1}$. Note that $[A_l,B^{ll}]=M^{ll}$.
Hence $B_{ll}^{11} =  ... =B_{ll}^{I_l,I_l}=
\Delta_{ll}=B_{ll}^{I_l+1,I_l+1} - a_l =  ...
=B_{ll}^{D_l,D_l}-a_l$. Since $B \in [A,End(V)]$ we have
$tr(B_{ll})=0$. Thus $\Delta_{ll} = \frac{D_l-I_l}{D_l}a_{l}$.

Since $B$ is nilpotent $\Delta$ is nilpotent. Hence $tr(\Delta)=0$
and thus $\sum_{l=1}^{k'} \frac{D_l-I_l}{D_l}a_{l}=0$ which means
$f(v_1,\phi_1,v_2,\phi_2)=0$.
\end{proof}

\appendix

\section{Theorem A implies Theorem B} \label{BtoA}
\setcounter{lemma}{0}

This appendix is analogous to section 1 in \cite{AGRS}. There, the
classical theory of Gelfand and Kazhdan (see \cite{GK}) is used.
Here we use an archimedean analog of this theory which is
described in \cite{AGS}, section 2. We will also use the theory of
nuclear \Fre spaces. For a good brief survey on this theory we
refer the reader to \cite{CHM}, Appendix A.

\begin{notation} $ $\\
(i) For a smooth \Fre representation $\pi$ of a real reductive group we denote by
$\widetilde{\pi}$  the smooth dual of $\pi$.\\
(ii) For a representation $\pi$ of $\mathrm{GL}_{n}(F)$ we let
$\widehat{\pi}$ be the representation of $\mathrm{GL}_{n}(F)$
defined by $\widehat{\pi}=\pi \circ \theta$, where $\theta$ is the
(Cartan) involution $\theta(g)= {g^{-1}}^t$.
\end{notation}

We will use the following theorem.

\begin{theorem}[Casselman - Wallach globalization]
Let $G$ be a real reductive group. 
There is a canonical equivalence of categories between the category of admissible smooth \Fre representations of $G$ and the category of
admissible $(\g,K)$- modules.
\end{theorem}
See e.g. \cite{Wal92}, chapter 11.

We will also use the embedding theorem of Casselman.

\begin{theorem}
Any irreducible $(\g,K)$-module can be imbedded into a
$(\g,K)$-module of principal series.
\end{theorem}

Those two theorems have the following corollary.

\begin{corollary}
The underlying topological vector space of any admissible smooth \Fre representation is a nuclear \Fre space.
\end{corollary}

\begin{definition}
Let $G$ and $H$ be real reductive groups. Let $(\pi,E)$ and
$(\tau,W)$ be admissible smooth \Fre representations of $G$ and
$H$ respectively. We define $\pi \otimes \tau$ to be the natural
representation of $G\times H$ on the space $E \widehat{\otimes}
W$.
\end{definition}

The following proposition is well known. For the benefit of the
reader we include its proof in subsection \ref{ProofIrr}.
\begin{proposition} \label{PropTensor}
Let $G$ and $H$ be real reductive groups. Let $\pi$ and $\tau$ be
irreducible admissible Harish-Chandra modules of $G$ and $H$
respectively. Then $\pi \otimes \tau$ is irreducible
Harish-Chandra module of $G \times H$.
\end{proposition}

\begin{corollary}\label{IrrProd}
Let $G$ and $H$ be real reductive groups. Let $\pi$ and $\tau$ be
irreducible admissible smooth \Fre representations of $G$ and $H$
respectively. Then $\pi \otimes \tau$ is irreducible
representation $G \times H$.
\end{corollary}

\begin{lemma}
Let $G$ and $H$ be real reductive groups. Let $(\pi,E)$ and
$(\tau,W)$ be admissible smooth \Fre representations of $G$ and
$H$ respectively. Then $Hom_{\C}(\pi,\tau)$ is canonically
embedded to $Hom_{\C}(\pi \otimes \widetilde{\tau},\C).$
\end{lemma}
\begin{proof}
For a nuclear \Fre space $V$ we denote by $V'$ its dual space
equipped with the strong topology. Let $\widetilde{W}$ denote the
underlying space of $\widetilde{\tau}$. By the theory of nuclear
\Fre spaces, we know $Hom_{\C}(E,W) \cong E' \widehat{\otimes} W$
and $Hom_{\C}(E \widehat{\otimes} \widetilde{W}, \C) \cong E'
\widehat{\otimes} \widetilde{W}'$. The lemma follows now from the
fact that $W$ is canonically embedded to $\widetilde{W}'$.
\end{proof}

We will use the following two archimedean analogs of theorems of
Gelfand and Kazhdan.

\begin{theorem}\label{GKreal}
Let $\pi$ be an irreducible admissible representation of
$\mathrm{GL}_{n}(F)$. Then $\widehat{\pi}$ is isomorphic to
$\widetilde{\pi}$.
\end{theorem}
For proof see \cite{AGS}, Theorem 2.4.4.

\begin{theorem}\label{DistCrit}
Let $H \subset G$ be real reductive groups and let $\sigma$ be an
involutive anti-automorphism of $G$ and assume that $\sigma(H)=H$.
Suppose $\sigma(\xi)=\xi$ for all $H$-bi-invariant Schwartz
distributions $\xi$ on $G$. Let $\pi$ be an irreducible admissible smooth \Fre representation of $G$. Then
$$\dim \Hom_{H}(\pi,\cc) \cdot \dim \Hom_{H}(\widetilde{\pi},\cc)\leq 1.$$
\end{theorem}
For proof see \cite{AGS}, Theorem 2.3.2.

\begin{corollary}
Let $H \subset G$ be real reductive groups and let $\sigma$ be an
involutive anti-automorphism of $G$ such that $\sigma(H)=H$.
Suppose $\sigma(\xi)=\xi$ for all Schwartz
distributions $\xi$ on $G$ which are invariant with respect to conjugation by $H$.

Let $\pi$ be an irreducible admissible smooth \Fre representation of $G$ and
$\tau$ be an irreducible admissible smooth \Fre representation of $H$. Then

$$dim \Hom_{H}(\pi,\tau) \cdot \dim \Hom_{H}(\widetilde{\pi},\widetilde{\tau})\leq 1.$$
\end{corollary}

\begin{proof}
Define $\sigma':G \times H \to G\times H$ by
$\sigma'(g,h):=(\sigma(g),\sigma(h))$. Let $\Delta H < G \times H$
be the diagonal. Consider the projection $G \times H \to H$. By
Frobenius reciprocity (Theorem \ref{Frob}), the assumption implies
that any $\Delta H$-bi-invariant distribution on $G \times H$ is
invariant with respect to $\sigma'$.

Hence by the previous theorem,  for any irreducible admissible
smooth \Fre representation $\pi'$ of $G \times H$ we have $\dim
\Hom_{\Delta H}(\pi',\cc) \cdot \dim \Hom_{\Delta
H}(\widetilde{\pi'},\cc)\leq 1.$

Taking $\pi':= \pi \otimes \widetilde{\tau}$ we obtain the required inequality.
\end{proof}

\begin{corollary}
Theorem A implies Theorem B.
\end{corollary}
\begin{proof}
By Theorem \ref{GKreal}, $\dim \Hom_{H}(\widetilde{\pi},\widetilde{\tau})
= \dim \Hom_{H}(\widehat{\pi},\widehat{\tau})=\dim \Hom_{H}(\pi,\tau)$.
\end{proof}

\subsection{Proof of proposition \ref{PropTensor}} \label{ProofIrr}

\begin{notation}
Let $G$ be a reductive group, $\g$ be its Lie algebra and $K$ be its maximal compact subgroup. Let $\pi$ be an admissible $(\g,K)$-module.\\
Let $\rho$ be an irreducible representation of $K$.\\
(i)  We denote by $e_{\rho}:\pi \to \pi$ the projection to the $K$-type $\rho$. \\
(ii) We denote by $G^{\pi}_\rho$ the subalgebra of $End(e_\rho(\pi))$ generated by the actions of $K$ and $e_\rho U(\g)e_\rho$.
\end{notation}

The following lemma is well-known

\begin{lemma} \label{erhoIrr}
Let $\pi$ be an irreducible admissible $(\g,K)$-module. Let $\rho$ be an irreducible representation of $K$.
Suppose that $e_\rho(\pi)\neq 0$. Then $e_\rho(\pi)$ is an irreducible representation of $G^{\pi}_\rho$.
\end{lemma}

We will also use Bernside theorem.

\begin{theorem}\label{Bernside}
Let $V$ be a finite dimensional complex vector space. Let $A
\subset End(V)$ be a subalgebra such that $V$ does not have any
non-trivial $A$-invariant subspaces. Then $A=End(V)$.
\end{theorem}

Now we are ready to prove proposition \ref{PropTensor}.

\begin{proof}[Proof of proposition \ref{PropTensor}]
Let $\g$ and $\h$ be the Lie algebras of $G$ and $H$. Let $K$ and
$L$ be maximal compact subgroups of $G$ and $H$. Let  $\omega
\subset \pi \otimes \tau$ be a nonzero $(\g \times \h, K\times
L)$-submodule. Then $\omega$ intersects non-trivially some
$K\times L$ type. Denote this type by $\rho \otimes \sigma$. By
Lemma \ref{erhoIrr}, $e_\rho(\pi)$ is an irreducible
representation of $G^{\pi}_\rho$ and $e_\sigma(\tau)$ is an
irreducible representation of $H^{\tau}_\sigma$. Hence by Bernside
theorem $G^{\pi}_\rho = End(e_\rho(\pi))$ and $H^{\tau}_\sigma =
End(e_\sigma(\tau))$. Hence $(G\times H)^{\pi \otimes \tau}_{\rho
\otimes \sigma} = End(e_\rho(\pi) \otimes e_\sigma(\tau))$. Thus
$\omega \cap e_{\rho \otimes \sigma}(\pi \otimes \tau)=e_{\rho
\otimes \sigma}(\pi \otimes \tau)$.

This means that $\omega$ contains an element of the form $v
\otimes w$, which implies that $\omega = \pi \otimes \tau$.
\end{proof}
\section{D-modules} \label{AppDmod}
\setcounter{lemma}{0}
In this appendix $X$ denotes a smooth affine variety defined over
$\R$. All the statements of this section extend automatically to
general smooth algebraic varieties defined over $\R$. In this
paper we use only the case when $X$ is an affine space.

\begin{definition}
Let $D(X)$ denote the algebra of polynomial differential operators on $X$.
We consider the filtration $F^{\leq i}D(X)$ on $D(X)$ given by the
order of differential operator.
\end{definition}

\begin{definition}
We denote by $\Gr D(X)$ the associated graded algebra of $D(X)$.

Define the symbol map $\sigma :D(X) \to \Gr D(X)$ in the following way.
Let $d \in D(X)$. Let $i$ be the minimal index such that $d \in F^{\leq i}$. We define $\sigma(d)$ to be the image of $d$ in
 $(F^{\leq i}D(X)) / (F^{\leq i-1}D(X))$
\end{definition}

\begin{proposition}
$\Gr D(X) \cong {\mathcal O}(T^*X)$.
\end{proposition}
For proof see e.g. \cite{Bor}. 

\begin{notation}
Let $(V,B)$ be a quadratic space. \\
(i) We define a morphism of algebras $\Phi^D_V: D(X \times V) \to
D(X \times V)$ in the following way.

Consider $B$ as a map $B:V \to V^*$. For any $f \in V^*$ we set
$\Phi^D_V(f) :=
\partial_{B^{-1}(f)}$. For any $v \in V$ we set $\Phi^D_V(\partial _v) :=
-B(v)$ and for any $d \in D(X)$ we set $\Phi^D_V(d):=d$. \\
(ii) It defines a morphism of algebras $\Phi^O_V:{\mathcal
O}(T^*X) \to {\mathcal O}(T^*X)$.
\end{notation}

The following lemma is straightforward.
\begin{lemma}
Let $f$ be a homogeneous polynomial. Consider it as a differential operator. Then $\sigma( \Phi^D_V(f))=\Phi^O_V(\sigma(f))$.
\end{lemma}

The D-modules we use in the paper are right D-modules. The
difference between right and left D-modules is not essential (see
e.g. section VI.3 in \cite{Bor}). We will use the notion of good
filtration on a D-module, see e.g. section II.4 in \cite{Bor}. Let
us now remind the definition of singular support of a module and a
distribution.

\begin{notation}
Let $M$ be a $D(X)$-module. Let $\alpha \in M$ be an element. Then we denote by $Ann_{D(X)}$ the annihilator of $\alpha$.
\end{notation}

\begin{definition}
Let $M$ be a $D(X)$-module. Choose a good filtration on $M$. Consider $grM$ as a module over $\Gr D(X) \cong {\mathcal O}(T^*X)$. We define
$$SS(M):= \Supp (\Gr M) \subset T^*X.$$
This does not depend on the choice of the good filtration on $M$
(see e.g. \cite{Bor}, section II.4).

For a distribution $\xi \in S^*(X(\R))$ we define $SS(\xi)$ to be
the singular support of the module of distributions generated by
$\xi$.
\end{definition}

The following proposition is trivial.

\begin{proposition}
Let $I < D(X)$ be a right ideal. Consider the induced filtrations
on $I$ and $D(X)/I$. Then $\Gr (D(X)/I) \cong \Gr (D(X)) / \Gr
(I)$.
\end{proposition}

\begin{corollary}
Let $\xi \in S^*(X)$. Then $SS(\xi)$ is the zero set of $\Gr (Ann_{D(X)} \xi)$.
\end{corollary}

\begin{corollary}
Let $I< {\mathcal O}(T^*X) $ be the ideal generated by
$\{\sigma(d) \, | \, d \in Ann_{D(X)}(\xi) \}$. Then $SS(\xi)$ is
the zero set of $I$.
\end{corollary}

\begin{corollary}
Fact \ref{Ginv} holds.
\end{corollary}

\begin{lemma}
Let $\xi \in S^*(X)$. Let $Z \subset X$ be a closed subvariety
such that $\Supp(\xi) \subset Z(\R)$. Let $f \in {\mathcal O}(X)$
be a polynomial that vanishes on $Z$. Then there exists $k \in \N
$ such that $f^k\xi=0$.
\end{lemma}
\begin{proof}$ $

Step 1. Proof for the case when $X$ is affine space and $f$ is a
coordinate function.\\
This follows from the proof of Corollary 5.5.4 in \cite{AG1}.

Step 2. Proof for the general case.\\
Embed $X$ into an affine space $A^N$ such that $f$ will be a
coordinate function and consider $\xi$ as distribution on $A^N$
supported in $X$. By Step 1, $f^k\xi=0$ for some $k$.
\end{proof}

\begin{corollary}
Fact \ref{Supp2SS} holds.
\end{corollary}

\begin{proposition}
Fact \ref{Fou} holds. Namely:

Let $(V,B)$ be a quadratic space. Let $Z \subset X \times V$ be a
closed subvariety, invariant with respect to homotheties in $V$.
Suppose that $\Supp(\xi) \subset Z(\R)$. Then $SS(\Fou_V(\xi))
\subset F_V(p_{X \times V}^{-1}(Z))$.
\end{proposition}
\begin{proof}
Let $f \in {\mathcal O}(X\times V)$ be homogeneous with respect to
homotheties in $V$. Suppose that $f$ vanishes on $Z$. Then
$\Phi^D_V(f^k) \in Ann_{D(X)} (\Fou_V(\xi))$. Therefore $\sigma
(\Phi^D_V(f^k))$ vanishes on $SS(\Fou_V(\xi))$. On the other hand,
$\sigma(\Phi^D_V(f^k)) =
\Phi^O_V(\sigma(f^k))=(\Phi^O_V(\sigma(f)))^k$. Hence
$SS(\Fou_V(\xi))$ is included in the zero set of
$\Phi^O_V(\sigma(f))$. Intersecting over all such $f$ we obtain
the required inclusion.
\end{proof}

\end{document}